\documentclass{amsart}

\usepackage{amssymb,amsmath,amscd,graphicx,amsthm}

\newtheorem{theorem}{Theorem}
\newtheorem{lemma}[theorem]{Lemma}
\newtheorem{proposition}[theorem]{Proposition}

\theoremstyle{definition}

\theoremstyle{remark}

\numberwithin{equation}{section}
\numberwithin{theorem}{section}

\def\A{{\mathcal A}}

\def\C{{\mathbb C}}

\def\FFF{{\mathfrak F}}
\def\G{{\mathcal G}}
\def\H{\mathcal H}

\def\P{{\mathcal P}}
\def\PP{{\mathfrak P}}

\def\Z{{\mathbb Z}}

\def\d{\mathbf d}

\def\g{\mathfrak g}

\def\ls{\le}

\def\one{\mathbf 1}

\def\wB{{\widetilde{B}}}

\def\wx{{\widetilde{\bf x}}}
\def\x{{\bf x}}

\def\End{\operatorname{End}}

\def\Poi{{\{\cdot,\cdot\}}}

\def\Tr{\operatorname{Tr}}

\def\diag{\operatorname{diag}}

\def\grad{\mbox{grad}}

\def\rank{\operatorname{rank}}
\def\s{\operatorname{sign}}

\def\:{{:\ }}

\begin{document}

\title[ Cluster algebra compatible  Poisson structures in Grassmannians]
{Poisson structures compatible with the  cluster algebra structure in Grassmannians}

\author{M. Gekhtman}

\address{Department of Mathematics, University of Notre Dame, Notre Dame,
IN 46556}
\email{mgekhtma@nd.edu}

\author{M. Shapiro}
\address{Department of Mathematics, Michigan State University, East Lansing,
MI 48823}
\email{mshapiro@math.msu.edu}

\author{A. Stolin}
\address{Department of Mathematics 
University of Goteborg 
411 24 Goteborg 
Sweden}
\email{mshapiro@math.msu.edu}
 
\author{A. Vainshtein}
\address{Department of Mathematics \& Department of Computer Science, University of Haifa, Haifa,
Mount Carmel 31905, Israel}
\email{alek@cs.haifa.ac.il}

\begin{abstract} 
We describe all Poisson brackets compatible with the natural cluster algebra structure
in the open Schubert cell of the Grassmannian $G_k(n)$ and show that any such bracket 
endows  $G_k(n)$ with a structure of a Poisson homogeneous space with respect to the 
natural action of $SL_n$ equipped with an R-matrix Poisson-Lie structure. The corresponding 
R-matrices belong to the simplest class in the Belavin-Drinfeld classification.
Moreover, every compatible Poisson structure can be obtained this way.

\end{abstract}


\maketitle

\section{Introduction}

The notion of a Poisson bracket {\em compatible with a cluster algebra structure\/} was introduced in
\cite{GSV1}. It was used to interpret cluster transformations and matrix mutations in cluster algebras 
from a viewpoint of Poisson geometry. In addition, it was shown in \cite{GSV1} that if a Poisson 
variety $\left ( \mathcal{M}, \Poi\right )$ possesses a  coordinate chart
that consists of regular functions whose logarithms have pairwise constant Poisson brackets, then 
one can use this chart to define a cluster algebra $\A_{\mathcal{M}}$, which 
is closely related (and, under rather mild conditions, isomorphic) to the ring of regular functions 
on $\mathcal{M}$, and such that $\Poi$ is compatible with $\A_{\mathcal{M}}$. This construction was 
applied to an open cell $G^0_k(n)$ in the Grassmannian $G_k(n)$ viewed as a Poisson homogeneous space 
under the action of $SL_n$ equipped with the {\em standard\/} Poisson-Lie structure. 
The resulting cluster algebra $\A_{G^0_k(n)}$  can be viewed as a restriction
of the cluster algebra structure in the coordinate ring of $G_k(n)$. 
This ``larger'' cluster algebra  $\A_{G_k(n)}$ was described in \cite{Scott} using 
combinatorial properties of Postnikov's map from the 
space of edge weights of a planar directed network into the Grassmannian \cite{Postnikov}. Poisson 
geometric properties of Postnikov's map are studied in \cite{GSV3}. One of the by-products of this study
is the existence of a {\em pencil\/} of Poisson structures on $G_k(n)$ compatible with $\A_{G^0_k(n)}$.
It turnes out that every bracket in the pencil defines a Poisson homogeneous structure with respect to
a {\em Sklyanin\/} Poisson-Lie bracket associated with a solution of the {\em modified classical 
Yang-Baxter equation (MCYBE)\/} of the form $R_t = R_0 + t A \pi_0$, where $t$ is a scalar parameter, 
$A$ is a certain fixed skew-symmetric $n\times n$ matrix, $R_0 = \pi_+ - \pi_-$, and $\pi_{\pm, 0}$ 
are projections onto strictly upper/strictly lower/diagonal part in $sl_n$ (the standard Poisson-Lie 
structure corresponds to $t=0$).

According to the {\em Belavin-Drinfeld classification\/} \cite{BD},
skew-symmetric solutions of MCYBE are defined by two types of data: discrete data associated 
with the  Dynkin diagram and called the {\em  Belavin-Drinfeld triple\/} and continuous
data associated with the Cartan subalgebra. We will say that two R-matrices belong to the same class
if the corresponding Belavin-Drinfeld triples are the same. R-matrices $R_t$ mentioned above 
belong to the class associated with the trivial Belavin-Drinfeld triple. The entire class 
consists of R-matrices of the form $R_S = R_0 + S \pi_0$, with $S$ arbitrary skew-symmetric. 
On the other hand, the Poisson pencil described above does not exhaust all Poisson structures
compatible with $\A_{G^0_k(n)}$. The main goal of this paper is to prove

\begin{theorem}
\label{mainintro}
The Poisson homogeneous structure with respect to the Poisson-Lie bracket associated with {\em any\/} 
$R_S$ is compatible with $\A_{G^0_k(n)}$. Moreover, up to a scalar multiple, all Poisson brackets 
compatible with $\A_{G^0_k(n)}$ are obtained this way.
\end{theorem}

It should be noted that Poisson brackets compatible with the ``larger'' cluster algebra  $\A_{G_k(n)}$ 
are naturally associated with Poisson structures on the {\em Grassmann cone\/} 
$\mathcal{C} (G_k(n))$ that can be realized as one-dimensional extensions of corresponding structures on
$G_k(n)$. Both the formulation and the proof of Theorem \ref{mainintro} can be modified in a rather 
straightforward way to the case of the cluster algebra $\A_{G_k(n)}$. A detailed description
of the relationship between $\A_{G_k(n)}$ and $\A_{G^0_k(n)}$ this modification relies upon
is presented in Chapter 4 of the forthcoming book \cite{book}.

The paper is organized as follows. In Section \ref{CA&PB} we recall the necessary information on 
cluster algebras and compatible Poisson structures and show how the latter can be completely described 
via the use of a toric action. Section \ref{sklya} provides the background on Poisson-Lie groups and 
Sklyanin brackets. Finally, in Section \ref{grass} we review Poisson homogeneous structures on 
$G_k(n)$ and the construction of $\A_{G^0_k(n)}$, and then proceed to prove Theorem \ref{mainintro}, 
which is re-stated there as  Theorem \ref{main}.

\section{Cluster algebras and compatible Poisson brackets}
\label{CA&PB}

We start with the basics on cluster algebras of geometric type. The definition that we present
below is not the most general one, see, e.g., \cite{FZ2, CAIII} for a detailed exposition. In what 
follows, we will use notation $[i,j]$ for an interval
$\{i, i+1, \ldots , j\}$ in $\mathbb{N}$, and write $[n]$ instead of $[1, n]$.
 
The {\em coefficient group\/} $\PP$ is a free multiplicative abelian
group of a finite rank $m$ with generators $g_1,\dots, g_m$.
An {\em ambient field\/}  is
the field $\FFF$ of rational functions in $n$ independent variables with
coefficients in the field of fractions of the integer group ring
$\Z\PP=\Z[g_1^{\pm1},\dots,g_m^{\pm1}]$ (here we write
$x^{\pm1}$ instead of $x,x^{-1}$).

A {\em seed\/} (of {\em geometric type\/}) in $\FFF$ is a pair
$\Sigma=(\x,\widetilde{B})$,
where $\x=(x_1,\dots,x_n)$ is a transcendence basis of $\FFF$ over the field of
fractions of $\Z\PP$ and $\widetilde{B}$ is an $n\times(n+m)$ integer matrix
whose principal part $B$ (that is, the $n\times n$ submatrix formed by the
columns $1,\dots,n$) is skew-symmetrizable. In this paper, we will only deal
with the case when $B$ is skew-symmetric.

The $n$-tuple  $\x$ is called a {\em cluster\/}, and its elements
$x_1,\dots,x_n$ are called {\em cluster variables\/}. Denote
$x_{n+i}=g_i$ for $i\in [m]$. We say that
$\widetilde{\x}=(x_1,\dots,x_{n+m})$ is an {\em extended
cluster\/}, and $x_{n+1},\dots,x_{n+m}$ are {\em stable
variables\/}. It is convenient to think of $\FFF$ as
of the field of rational functions in $n+m$ independent variables
with rational coefficients. 

Given a seed as above, the {\em adjacent cluster\/} in direction $k\in [n]$
is defined by
$$
\x_k=(\x\setminus\{x_k\})\cup\{x'_k\},
$$
where the new cluster variable $x'_k$ is given by the {\em exchange relation}
\[
x_kx'_k=\prod_{\substack{1\le i\le n+m\\  b_{ki}>0}}x_i^{b_{ki}}+
       \prod_{\substack{1\le i\le n+m\\  b_{ki}<0}}x_i^{-b_{ki}};
\]
here, as usual, the product over the empty set is assumed to be equal to~$1$.

We say that $\wB'$ is obtained from $\wB$ by a {\em matrix mutation\/} in direction $k$
and write $\wB'=\mu_k(\wB)$ 
 if
\[
b'_{ij}=\begin{cases}
         -b_{ij}, & \text{if $i=k$ or $j=k$;}\\
                 b_{ij}+\displaystyle\frac{|b_{ik}|b_{kj}+b_{ik}|b_{kj}|}2,
                                                  &\text{otherwise.}
        \end{cases}
\]

Given a seed $\Sigma=(\x,\widetilde{B})$, we say that a seed
$\Sigma'=(\x',\widetilde{B}')$ is {\em adjacent\/} to $\Sigma$ (in direction
$k$) if $\x'$ is adjacent to $\x$ in direction $k$ and $\widetilde{B}'=
\mu_k(\widetilde{B})$. Two seeds are {\em mutation equivalent\/} if they can
be connected by a sequence of pairwise adjacent seeds.

The {\em cluster
algebra\/} (of {\em geometric type\/}) $\A=\A(\widetilde{B})$
associated with $\Sigma$ is the $\Z\PP$-subalgebra of $\FFF$
generated by all cluster variables in all seeds mutation
equivalent to $\Sigma$.

Let $\Poi$ be a Poisson bracket on the ambient field $\FFF$
considered as the field of rational functions in $n+m$ independent
variables with rational coefficients. We say that it is {\em compatible} with 
the cluster algebra $\A$ if, for any extended
cluster $\widetilde{\x}=(x_1,\dots,x_{n+m})$,  one has
$$
\{x_i,x_j\}=\omega_{ij} x_ix_j,
$$
where $\omega_{ij}\in\Z$ are
constants for all $i,j\in[n+m]$. The matrix
$\Omega^{\widetilde \x}=(\omega_{ij})$ is called the {\it coefficient matrix\/}
of $\Poi$ (in the basis $\widetilde \x$); clearly, $\Omega^{\widetilde \x}$ is
skew-symmetric.

In what follows, we denote by $A(I,J)$ the submatrix of a matrix $A$ with a row set $I$ 
and a column set $J$. Consider, along with cluster and stable variables $\wx$, another $(n+m)$-tuple of
rational functions denoted $\tau=(\tau_1,\dots,\tau_{n+m})$ and defined by
\begin{equation}\label{eq:1.3}
\tau_j=x_j^{\varkappa_j}\prod_{k=1}^{n+m} x_k^{b_{jk}},
\end{equation}
where $\widehat{B}=(b_{jk})_{j,k=1}^{n+m}$ is a fixed skew-symmetric matrix such that 
$\widehat{B}([n], [n+m]) = \widetilde{B}$, $\varkappa_j$ is an integer, $\varkappa_j=0$ for $1\ls j\ls n$.
We say that the entries $\tau_i$, $i\in[n+m]$, form a \emph{$\tau$-cluster}. It is proved in 
\cite{GSV1}, Lemma 1.1, that $\varkappa_j$, $ n+1 \leq j \leq n+m$, can be selected in such a way that
the transformation $\wx\mapsto\tau$ is  non-degenerate, provided 
$\rank \wB=n$. We denote $\varkappa=\diag(\varkappa_i)_{i=1}^{n+m}$  and $B_\varkappa=\widehat{B}+
\varkappa$.  Nondegeneracy of the transformation $\wx \mapsto
\tau$ is equivalent to nondegeneracy of $B_\varkappa$.

Recall that a square matrix $A$ is {\it reducible\/} if there exists a
permutation matrix $P$ such that $PAP^T$ is a block-diagonal
matrix, and {\it irreducible\/} otherwise.
The following result is a particular case of Theorem 1.4 in \cite{GSV1}.

\begin{theorem}\label{thm:1.4}
Assume that $\rank \wB=n$ and the principal part of $\wB$ is irreducible.
Then a Poisson bracket is compatible with $\A(\wB)$ if and only if its coefficient matrix 
$\Omega^\tau$ in the basis $\tau$ has the following property: the submatrix  
$\Omega^\tau([n],[n+m])$ is proportional to $\wB$.
\end{theorem}

Starting with an arbitrary $\Poi_0^\A$ compatible with $\A$, one can suggest an alternative 
description of all other compatible Poisson brackets via the following construction. 
Let  $C=(c_{ij})$ be an integral $(n+m)\times m$ matrix. Define an action of 
$(\C^*)^m=\{ \d=(d_1,\ldots, d_m)\ : \  d_1 \cdots d_r \ne 0 \}$ on $\wx$ by 
\begin{equation}
\d.\wx=\left ( x_i \prod_{\alpha=1}^m d_\alpha^{c_{i\alpha}}\right )_{i=1}^{n+m}.
\label{toricact}
\end{equation}
We say that \eqref{toricact} {extends to an action of $(\C^*)^m$ on $\A$} if the action 
induced by it in {\em any} cluster in $\A$ has a form  \eqref{toricact} (with possibly different 
coefficients $c_{i\alpha}$). Lemma 2.3 in \cite{GSV1} claims that \eqref{toricact} extends to an 
action of $(\C^*)^m$ on $\A$ if and only if $\wB C = 0$. 
The same condition guarantees  that  $\tau_i(\d.\wx)=\tau_i(\wx)$ for
$i\in [n]$.  Since $B_\varkappa$ is invertible, any such $C$ of full rank has a form 
$C=B_\varkappa^{-1}([n+m],[n+1,n+m]) U$, where $U$ is any invertible $m\times m$ matrix.

Next, assume that $(\C^*)^m$ is equipped with a Poisson structure given by
$$
\{d_i, d_j\}_V = v_{ij} d_i d_j,
$$
where $V=(v_{ij})$ is a fixed skew-symmetric matrix. 

\begin{proposition} Let $\wB$ satisfy the assumptions of 
Theorem~\ref{thm:1.4}.  
Then for any skew-symmetric $m\times m$ matrix $V$, there exists a  
Poisson structure $\Poi_V^\A$ compatible with  
$\A$ such that the map 
$\left((\C^*)^m \times \A, \Poi_V \times \Poi_0^\A \right)\to \left( \A, \Poi_V^\A\right)$ 
extended from the action $(\d,\wx) \mapsto \d.\wx$ is Poisson. Moreover, every compatible 
Poisson bracket on $\A$ is a scalar multiple of $\Poi_V^\A$ for some $V$.
\label{allcompat}
\end{proposition}

\begin{proof}  Let $\Omega^{\wx}$ be the coefficient matrix of $\Poi_0^\A$ in the basis $\wx$.
It is easy to see that in the product structure $\Poi_V \times \Poi_0^\A$ on
$(\C^*)^m \times \A$,
$$
\{(\d.\wx)_i, (\d.\wx)_j\} =\left  (\Omega^{\wx} + C V C^T\right)_{ij}(\d.\wx)_i (\d.\wx)_j. 
$$
Thus, for the action  $(\d,\wx) \mapsto \d.\wx$ to be Poisson, one must have 
$\{x_i,x_j\}^\A_V=  \left (\Omega^{\tilde\x} + C V C^T\right )_{ij} x_i x_j$ for $i,j \in [n+m]$. Since
$\tau_i(\d.\wx)=\tau_i(\wx)$ if $i\in [n]$, and $\tau_i(\d.\wx)=\tau_i(\wx) m_i(\d)$ for some 
monomials $ m_i(\d)$ in $\d$ for $i\in [n+1, n+m]$, we see that $\{\tau_i, \tau_j\}^\A_V = 
\{\tau_i, \tau_j\}^\A_0$ for $i\in [n]$, $j\in [n+m]$. Since  $\Poi_0^\A$ is a compatible 
Poisson bracket, Theorem \ref{thm:1.4} yields that $\Poi_V^\A$ is compatible as well. 

Now, let $\Omega^\tau$ be the coefficient matrix of $\Poi_0^\A$ in the basis $\tau$. Denote
$Z_0=\Omega^\tau([n+1,n+m],[n+1,n+m])$.
Consider $\{\tau_i, \tau_j\}^\A_V$ for $i,j\in [n+1, n+m]$:  
$\{\tau_i, \tau_j\}^\A_V=  z_{ij}\tau_i \tau_j$. To compute $Z=(z_{ij})_{i,j=n+1}^{n+m}$, note 
that the matrix that describes $\Poi^\A_V$ in coordinates $\tau$ is  $\Omega_V^\tau = B_\varkappa  
\left (\Omega^{\wx} + C V C^T\right ) B_\varkappa^T$, and thus
$$
Z=\Omega_V^\tau([n+1,n+m], [n+1,n+m]) = Z_0 + U V U^{-1}. 
$$
It is clear that by varying $V$, one can 
make $Z$ to be equal to an arbitrary skew-symmmetric $m\times m$ matrix.
Theorem~\ref{thm:1.4} implies that up to a scalar multiple, the matrix block $Z$ determines a 
compatible Poisson structure uniquely, and the result follows.
\end{proof}

\section{Poisson-Lie groups and Sklyanin brackets}
\label{sklya}

We need to recall some facts about Poisson-Lie groups
(see, e.g.\cite{r-sts}).

Let $\G$ be a Lie group equipped with a Poisson bracket $\Poi$.
$\G$ is called a {\em Poisson-Lie group\/} if the multiplication map
$$
{\mathfrak m} : \G\times \G \ni (x,y) \mapsto x y \in \G
$$
is Poisson. Perhaps, the most important class of Poisson-Lie groups
is the one associated with classical R-matrices. 

Let $\g$ be a Lie algebra of $\G$. Assume that $\g$ is equipped with a nondegenerate invariant
bilinear form $(\ ,\ )$. An element $R\in \End(\g)$ is a {\em classical R-matrix\/} if it is 
a skew-symmetric operator that satisfies the {\em modified classical Yang-Baxter equation\/} (MCYBE)
\begin{equation}
[R(\xi), R(\eta)] - R \left ( [R(\xi), \eta]\ + \ [\xi, R(\eta)] \right ) = - [\xi,\eta].
\label{MCYBE}
\end{equation}

Given a classical R-matrix $R$, $\G$ can be endowed with a Poisson-Lie structure as follows.
Let $\nabla f$, $\nabla' f$ be the right and the left gradients for a function $f\in C^\infty(\G)$:
\begin{equation}
( \nabla f ( x) , \xi ) = \frac{d} {dt} f( \exp{(t \xi)} x)\vert_{t=0},\qquad
 (\nabla' f ( x) , \xi ) = \frac{d} {dt} f( x\exp{(t \xi)})\vert_{t=0}.
\label{rightleftgrad}
\end{equation}
Then the  bracket given by 
\begin{equation}\label{Rbra}
\{ f_1, f_2\} = \frac{1}{2} ( R(\nabla' f_1), \nabla' f_2 ) -
\frac{1}{2} ( R(\nabla f_1), \nabla f_2 )\ 
\end{equation}
is a Poisson-Lie bracket on $\G$ called the {\em Sklyanin bracket}.

We are interested in the case 
$\G=SL_n$ and $\g=sl_n$ equipped with the trace-form
$$
(\xi, \eta) = \Tr ( \xi \eta).
$$
In this case, the right and left gradients~(\ref{rightleftgrad}) are
$$
\nabla f(x) = x\ \grad f (x), \qquad \nabla' f(x) = \grad f (x)\ x,
$$
where
$$
\grad f (x) = \left (  \frac{\partial f} {\partial x_{ji}} \right
)_{i,j=1}^n,
$$
and the Sklyanin bracket becomes
\begin{multline}
\{ f_1, f_2\}_{R}(x) = \\
\frac{1}{2} ( R(\grad f_1(x)\ x), \grad
f_2(x)\ x ) -  \frac{1}{2} ( R(x\ \grad f_1(x)), x\ \grad f_2(x) ). \label{sklyaSLn}
\end{multline}

Every $\xi \in \mathfrak g$ can be uniquely decomposed as 
\begin{equation}\nonumber
\xi = \pi_-(\xi) + \pi_0(\xi)  + \pi_+(\xi),
\label{decomposition_algebra}
\end{equation}
where $\pi_+(\xi)$ and $\pi_-(\xi)$ are strictly upper and lower triangular  and $\pi_0(\xi) $
is diagonal.
The simplest classical R-matrix on $sl_n$ is given by
\begin{equation}
R_0 (\xi) = \pi_+(\xi) - \pi_-(\xi)= \left ( \s(j-i) \xi_{ij}\right )_{i,j=1}^n.
\label{standardR}
\end{equation}
 Substituting $R=R_0$ into~(\ref{sklyaSLn}), we find the bracket for a pair of matrix entries: 
\begin{equation}\label{braijkl}
\{ x_{ij}, x_{i'j'}\}_{R_0}=\frac{1}{2} \left ( \s(i'-i) +
\s(j'-j)\right ) x_{ij'} x_{i'j}.
\end{equation}

It is known (see \cite{r-sts}) that if $R_0$ is the standard R-matrix,
$S$ is any linear operator on the space of traceless diagonal matrices that is skew-symmetric with
respect to the trace-form, and $\pi_0$ is the natural projection onto the subspace of 
diagonal matrices, then
\begin{equation}
\label{Rs}
R_S = R_0 + S \pi_0
\end{equation}
satisfies MCYBE~(\ref{MCYBE}), and thus gives rise to a Sklyanin Poisson-Lie bracket.
The operator $S$ can be identified with an $n\times n$ skew-symmetric matrix 
whose kernel contains the vector $(1,\ldots,1)$ and thus  is uniquely determined  by its 
$(n-1)\times (n-1)$ submatrix $(s_{ij})_{i,j=1}^{n-1}$, which we will also denote by $S$. Slightly
abusing notation, we denote the remaining elements of the above  $n\times n$ skew-symmetric matrix by
\[
s_{in}=-\sum_{j=1}^{n-1}s_{ij},\qquad s_{nj}=-\sum_{i=1}^{n-1}s_{ij}.
\]
The Sklyanin bracket \eqref{sklyaSLn} that corresponds to \eqref{Rs} can be written in terms of 
matrix entries as
\begin{equation}\label{braijklRS}
\{ x_{ij}, x_{i'j'}\}_{R_S}=
\{ x_{ij}, x_{i'j'}\}_{R_0}+
\frac{1}{2} \left (s_{i i'} - s_{j j'} \right ) x_{ij} x_{i'j'}.
\end{equation}

Let $\H$ denote the subgroup of diagonal matrices in $SL_n$:
$$
\H=\{\diag(d_1,\dots,d_n)\ : d_1\cdots d_n =1 \}.
$$
For any skew-symmetric matrix $V=(v_{ij})_{i,j=1}^{n-1}$, define a Poisson bracket $\Poi^\H_V$ 
on $\H$ by 
\begin{equation}
\label{diag}
\{d_i, d_j\}^\H_V = v_{ij} d_i d_j;
\end{equation}
here $v_{in}$ and $v_{nj}$ have the same meaning as $s_{in}$ and $s_{nj}$ above.

In what follows, we denote the Poisson manifolds $\left (\H, \Poi^\H_V \right )$ and
$\left (SL_n, \Poi_{R_S} \right )$ by $\H^{\{V\}}$ and $SL_n^{\{S\}}$, respectively.

Next, for $S$ defined as in \eqref{Rs}, consider the direct product of Poisson manifolds
$\H^{\{\frac 12 S\}} \times SL_n^{\{0\}} \times \H^{\{-\frac 12 S\}}$;
the product structure
we denote below simply by $\Poi$.

\begin{lemma} The map 
$\H^{\{\frac 12 S\}} \times SL_n^{\{0\}} \times \H^{\{-\frac 12 S\}} \to  SL_n^{\{S\}}$ 
given  by $( D_1, X, D_2) \mapsto D_1 X  D_2$  is Poisson.
\label{lemma}
\end{lemma}

\begin{proof} Denote $D_1 X D_2$ by $\widehat{X} = (\hat x_{ij})_{i,j=1}^n$. 
Let $D_k=\diag(d_{kl})_{l=1}^n$ for $k=1,2$. Then
$\{ \hat x_{ij}, \hat x_{i'j'} \} = 
\{ \hat x_{ij}, \hat x_{i'j'} \}_{R_0}
+  x_{ij}  x_{i'j'} \{d_{1i} d_{2j}, d_{1i'} d_{2j'}\} $. 
The second term is equal to 
\[
\frac{1}{2} ( s_{ii'} - s_{jj'}) x_{ij}  x_{i'j'} 
d_{1i} d_{2j} d_{1i'} d_{2j'} = \frac{1}{2} ( s_{ii'} - s_{jj'}) \hat x_{ij}  \hat x_{i'j'},
\]   
and the claim follows by \eqref{braijklRS}.
\end{proof}

\section{ Grassmannians}
\label{grass}

\subsection{} Let $\P$ be a Lie subgroup of a Poisson-Lie group $\G$. A Poisson
structure on the homogeneous space $\P \backslash \G$ is called
{\it Poisson homogeneous\/} (with respect to the Poisson-Lie structure on $\G$) \cite{D} 
if the action map $\P \backslash \G
\times \G \to \P \backslash \G$ is Poisson. In particular, if $\P$
is a parabolic subgroup of a simple Lie group $\G$ equipped with
the standard Poisson-Lie structure, then $\P \backslash \G$ is a
Poisson homogeneous space. We will be
interested in the case when $\G=SL_n$ equipped with the bracket
(\ref{sklyaSLn}) and
$$
\P=\P_k = \left \{ \begin{pmatrix} A & 0\\ B & C \end{pmatrix}\ :
A\in GL_k,  C\in GL_{n-k} \right \}.
$$
The resulting homogeneous space is the Grassmannian $G_k(n)$ equipped with what we will call
{\em the standard Poisson homogeneous structure\/} $\Poi_0^{Gr}$.
We will recall an explicit expression of this Poisson structure on the open Schubert cell  
$G^0_k(n)=\{ X\in G_k(n) : x_{[k]} \ne 0 \}$. Here we use the same notation for an element of 
the Grassmannian and its matrix representative $X$, and $x_I$ denotes the Pl\"ucker coordinate 
that corresponds to a $k$-element subset $I \subset [n]$. 
Elements of $G^0_k(n)$ can be represented by matrices of the form $\left [ \one_k \  Y\right ]$, 
and the entries of the $k\times(n-k)$ matrix $Y$ serve as coordinates on $G^0_k(n)$.
In terms of matrix elements $y_{ij}$ of $Y$, the Poisson homogeneous bracket  looks
as follows \cite{GSV1}:
\begin{equation}\label{Poihomstand}
\{y_{ij}, y_{\alpha\beta}\}_0^{Gr}=\frac{\s (\alpha - i)- \s (\beta -j)}{2}
y_{i\beta} y_{\alpha j}.
\end{equation}
%
We denote $G_k(n)$ equipped with the Poisson bracket \eqref{Poihomstand} by $G_k(n)^{\{ 0\}}$.

\begin{proposition}
\label{GrS}
{\rm (i)} For an arbitrary skew-symmetric operator $S$, there exists a Poisson bracket 
$\Poi^{Gr}_S$ on $G_k(n)$, unique up to a scalar multiple, such that the natural action  
$(X, D) \mapsto X D$ is a Poisson map from $G_k(n)^{\{ 0\}} \times  \H^{\{-\frac{1}{2} S\}} $ to 
$G_k(n)^{\{ S\}}:=\left (G_k(n), \Poi^{Gr}_S \right )$.

{\rm (ii)} The bracket $\Poi^{Gr}_S$ is
a Poisson homogeneous structure on $G_k(n)$  with respect to the bracket $\Poi_{R_S}$ on 
$SL_n$ defined by \eqref{braijklRS}.
\end{proposition}

\begin{proof} (i) Let $X=\left [ \one_k \  Y\right ]\in G^0_k(n)$, $D=\diag (d_1,\dots,d_n) \in \H$ 
and let $\left [ \one_k \  \widetilde{Y}\right ]$ be the matrix that represents the element 
$X D \in G^0_k(n)$. Then $\tilde y_{ij} = y_{ij} {d_{j+k}}/{d_i}$,
and the Poisson bracket of any two Pl\"ucker coordinates $\tilde y_{ij}$ and $\tilde y_{\alpha\beta}$ 
in the product structure $ \Poi^{Gr}_0 \times \Poi^\H_{-\frac{1}{2} S}$ is equal to
$$
\frac{\s (\alpha - i)- \s (\beta -j)}{2}
\tilde y_{i\beta} \tilde y_{\alpha j} + \frac{s_{i, \beta+k} + 
s_{j+k, \alpha} - s_{i, \alpha} - s_{j+k, \beta+k}}{2} \tilde y_{ij} \tilde y_{\alpha\beta}.
$$
Thus, the bracket defined on $G^0_k(n)$ by the formula
$$
\{y_{ij}, y_{\alpha\beta}\}_S^{Gr}=
\{y_{ij}, y_{\alpha\beta}\}_0^{Gr}
+ \frac{s_{i, \beta+k} + s_{j+k, \alpha} - s_{i, \alpha} - s_{j+k, \beta+k}}{2}y_{ij}  y_{\alpha\beta}
$$
is the unique, up to a scalar multiple, Poisson  bracket that makes the map $ (X, D) \mapsto XD $ 
Poisson. Sinse $G^0_k(n)$ is an open dense subset in $G_k(n)$ the claim follows.

(ii) To see that $\Poi^{Gr}_S$ is Poisson homogeneous  with respect to $\Poi_{R_S}$, 
we need to check that the natural action of $SL_n$ on $G_k(n)$ defines a Poisson map
from $G_k(n)^{\{ S\}}\times SL_n^{\{ S\}}$ to $G_k(n)^{\{ S\}}$. 
Instead of a straightforward calculation, we can use the fact that this is true for $S=0$ and 
Lemma \ref{lemma}. Indeed, it is easy to check that both Poisson maps 
$G_k(n)^{\{0\}}\times\H^{\{-\frac12 S\}}\to G_k(n)^{\{S\}}$ given by
$(X,D_1)\mapsto XD_1$ and 
$\H^{\{\frac 12 S\}} \times SL_n^{\{0\}} \times \H^{\{-\frac 12 S\}} \to  SL_n^{\{S\}}$ 
given  by $( D_1, X, D_2) \mapsto D_1 X  D_2$ are surjective. Therefore, we can replace the map
$G_k(n)^{\{ S\}}\times SL_n^{\{ S\}}\to G_k(n)^{\{ S\}}$ by the map
$$
G_k(n)^{\{0\}}\times\H^{\{-\frac12 S\}}\times\H^{\{\frac 12 S\}} \times SL_n^{\{0\}} \times 
\H^{\{-\frac 12 S\}} \to  G_k(n)^{\{ S\}}
$$
given by $(X,D_1,D_2,A,D_3)\mapsto XD_1D_2AD_3$. It is easy to check that 
$(D_1,D_2)\mapsto D_1 D_2$ is a Poisson map from
$\H^{\{-\frac{1}{2} S\}} \times \H^{\{\frac{1}{2} S\}}$ onto $\H^{\{0\}}$,
which, in turn, is a Poisson-Lie subgroup of $SL_n^{\{0\}}$. We thus arrive to the Poisson map
$G_k(n)^{\{0\}}\times SL_n^{\{0\}} \times \H^{\{-\frac 12 S\}} \to  G_k(n)^{\{ S\}}$ given by
$(X,\widetilde{A},D_3)\mapsto X\widetilde{A}D_3$ with $\widetilde{A}=D_1D_2A$. The standard Poisson 
homogeneous structure ensures that the map $G_k(n)^{\{0\}}\times SL_n^{\{0\}}\to  G_k(n)^{\{ 0\}}$ is
Poisson, and it remains to use part (i) of Proposition~\ref{GrS} to complete the proof. 
\end{proof}

\subsection{}
Now, we recall the construction
of the cluster algebra $\A_{G^0_k(n)}$ associated  with the open cell  $G^0_k(n)$ as 
described in \cite{GSV1,GSV3}.
Denote $m=n-k$. For every $i\in [k]$, $j\in [m]$ put
\begin{equation}\label{Iij}
I_{ij}
=\begin{cases} [i+1,k] \cup
[ j+k,i + j+ k - 1], &\text{if $i \le m-j +1$}\\
                     \left ([k]\setminus [i + j - m ,\ i]\right ) 
\cup [ j+k, n], &\text{if $i > m-j +1$}.
\end{cases}
\end{equation}
Denote the Pl\"ucker coordinate $x_{I_{ij}}$ by $x(i,j)$.

The initial cluster of $\A_{G^0_k(n)}$ is given by
\begin{equation}\label{initPluck}
\mathbf x =\mathbf x (k,n)=\left \{ 
\frac{x(i,j)}{x_{[k]}}\ :\ i\in [k],\ j\in [m]\right \}.
\end{equation}
 Stable variables  are
${\displaystyle
\frac{x(1,1)}{x_{[k]}}, \ldots, \frac{x(k,1)}{x_{[k]}}, \frac{x(k,2)}{x_{[k]}},\ldots, \frac{x(k,m)}{x_{[k]}}}$.
The entries of $\wB$ that correspond to $\x$ are all $0$ or $\pm 1$s. Thus it is convenient to 
describe $\wB$ by a directed  graph $\Gamma(\wB)$. 

\begin{figure}[ht]
\begin{center}
\includegraphics[height=3.5cm]{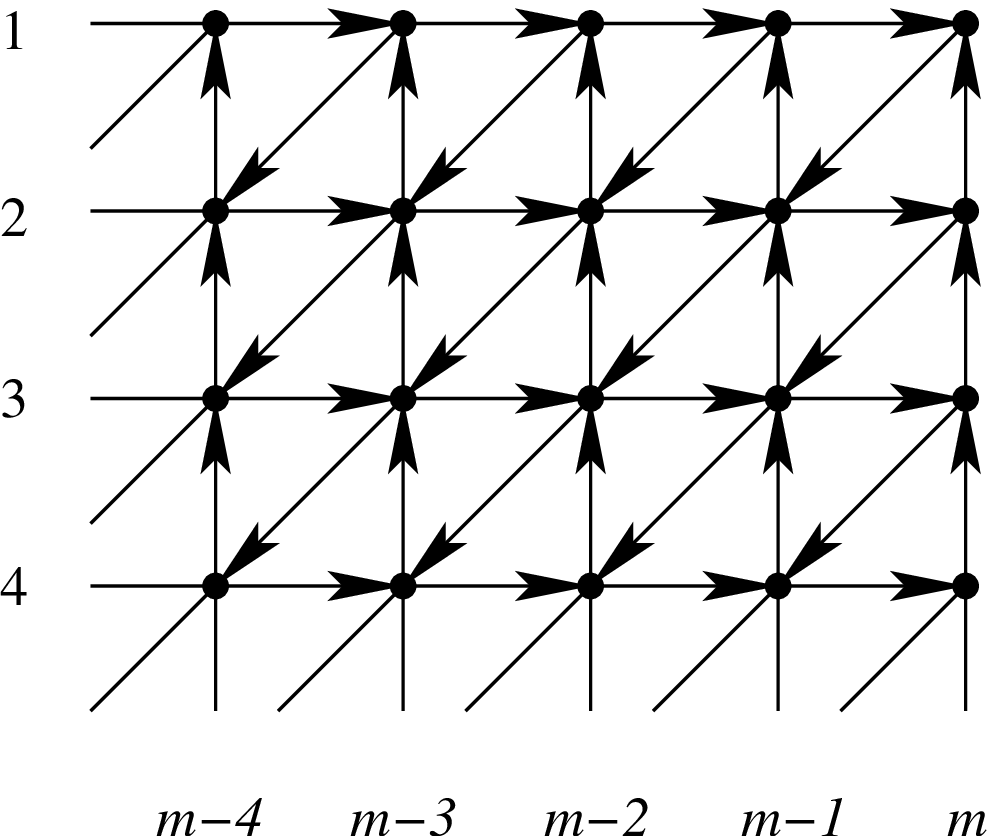}
\caption{Graph $\Gamma(\wB)$}\label{fig:graph_1}
\end{center}
\end{figure}

The
vertices of $\Gamma(\wB)$ correspond to all columns of $\wB$, and, since $\wB$ is rectangular, 
the corresponding edges are either  between the cluster variables or between a cluster variable 
and a stable variable. In our case, $\Gamma(\wB)$ is a directed graph
with $k m $ vertices labeled
by pairs of integers $(i,j)\ i\in[k], j\in [m]$.  $\Gamma(\wB)$ has edges $(i,j) \to (i,j+1)$,
$(i+1,j) \to (i,j)$ and $(i,j) \to (i+1,j-1)$ whenever both vertices defining an edge are in 
the vertex set of $\Gamma(\wB)$ 
(cf.~Fig.~\ref{fig:graph_1}). Each cluster variable $x(i,j)$ is associated
with (placed at) the vertex with coordinates $(i,j)$ of the grid in Fig.~\ref{fig:graph_1} for
$i\in [k],\ j\in [ m]$.
Equation \eqref{eq:1.3} then results in the following formulas for the $\tau$-cluster:
 \begin{equation}
\label{tau}
\tau_{ij}= \frac{x(i+1,j-1) x(i,j+1) x(i-1,j) } {x(i-1,j+1) x(i,j-1) x(i+1,j) }, \quad 
i\in[ k-1],\ j \in [2,m],
\end{equation}
where $x(0,j)=x(i,m+1) =1$.

\begin{lemma}
\label{tauinv} Functions \eqref{tau} are invariant under the natural action of $\H$ on $G_k(n)$.
\end{lemma}

\begin{proof}  Let $X\in G_k(n)$, $D=\diag(d_1,\dots,d_n) \in \H$ and $\widetilde{X} = X D$. 
For any  subset $I=\{i_1, \ldots, i_l\} \subset [n]$ denote $d^I = \prod_{j=1}^l d_{i_j}$. Then, using
\eqref{Iij}, we obtain
$$
\tilde x(i,j)=x(i,j) d^{I_{ij}} = \begin{cases} 
{\displaystyle
x(i,j) \frac{d^{[k]} d^{[i+j+k-1]} }{d^{[i]} d^{[j+k-1]}}} , &\text{if $i \le m-j +1$},\\[.7em]
{\displaystyle
x(i,j)\frac{d^{[k]}  d^{[n]} d^{[i+j-m-1]} }{d^{[i]} d^{[j+k-1]}}}, &\text{if $i > m-j +1$},
\end{cases}
$$
and the equality $\tau_{ij}(\widetilde{X})=\tau_{ij}(X)$ follows from \eqref{tau} 
by trivial cancellation.
\end{proof}

Now we are ready to prove
\begin{theorem}
\label{main}
A Poisson structure $\Poi$ on $G_k(n)$ is compatible with $\A_{G^0_k(n)}$ if and only if a 
scalar multiple of
$\Poi$ defines a Poisson homogeneous structure with respect to $\Poi_{R_S}$ for some skew-symmetric 
operator $S$.
\end{theorem} 

\begin{proof} It follows from Theorem 5.4 in \cite{GSV3} that 
$\Poi_0^{Gr}$ is compatible with $\A_{G^0_k(n)}$.
The number of stable variables for $\A_{G^0_k(n)}$ is $n-1$. 
Since $\H$ is isomorphic to $(\C^*)^{n-1}$, Lemma \ref{tauinv} guarantees 
that the map $ (X, D) \mapsto XD $ translates into an action of 
$(\C^*)^{n-1}$ on $\A_{G^0_k(n)}$ as described in Section \ref{CA&PB}. 
Assumptions of Theorem~\ref{thm:1.4} are verified in \cite{GSV1}, 
Section 3.3. Then Proposition \ref{allcompat} and Proposition \ref{GrS} 
imply that every compatible Poisson bracket on $\A_{G^0_k(n)}$ is a scalar  multiple of 
$\Poi_S^{Gr}$ for some  skew-symmetric operator $S$ on the space of traceless
diagonal $n\times n$ matrices. Since $\Poi_S^{Gr}$ is a unique 
Poisson homogeneous  with respect to $\Poi_{R_S}$ (see, e.g. \cite{D}), the claim follows.
\end{proof}

\section*{Acknowledgments}

M.~G.~was supported in part by NSF Grant DMS \#0801204. 
M.~S.~was supported in part by NSF Grants DMS \#0800671 and PHY \#0555346.  
A.~S.~was supported in part by KVVA. 
A.~V.~was supported in part by ISF Grant \#1032/08.
The authors are grateful to A.~Zelevinsky for useful comments.

 \end{document}